\documentclass[11pt]{article}
\usepackage{amsmath,amsthm,amssymb,comment,amsfonts,eucal,amscd}
\providecommand{\keywords}[1]{\textbf{\textit{Keywords---}} #1}
\begin{document}
\newtheorem{thm}{Theorem}[section]
\newtheorem{cor}[thm]{Corollary}
\newtheorem{lem}[thm]{Lemma}
\newtheorem{coranddef}[thm]{Corollary and Definition}
\newtheorem*{conj}{Conjecture}
\newtheorem*{thm*}{Theorem}
\theoremstyle{definition}
\newtheorem{defn}[thm]{Definition}
\newtheorem{nota}[thm]{Notation}
\newtheorem*{conv*}{Convention}
\newtheorem{exa}[thm]{Example}
\newtheorem*{subl*}{Sublemma}
\theoremstyle{remark}
\newtheorem{rem}[thm]{Remark}
\newcommand{\lm}{\lambda}
\newcommand{\ep}{\epsilon}
\newcommand{\sg}{\sigma}
\newcommand{\oF}{\overline{F}}
\newcommand{\bbc}{\mathbb{C}}
\newcommand{\bbz}{\mathbb{Z}}
\newcommand{\mnn}{\ensuremath { M_{n,n}}}
\newcommand{\jtr}{\mbox{$J$-trace}}
\newcommand{\superc}{supercommutative}
\newcommand{\sigc}{\sigma C_1^{\ep_1}\cdots C_d^{\ep_d}}
\newcommand{\xd}{x_1,\ldots,x_d}
\newcommand{\vd}{V^{\otimes d}}
\newcommand{\ootimes}{\otimes\cdots\otimes}
\def\noqed{\renewcommand{\qedsymbol}{}}
\title{$J$-Trace Identities and Invariant Theory}
\author{Allan Berele\\ Department of Mathematics\thanks{Partially supported by a Faculty Research Grant from DePaul University}\\ DePaul University\\ Chicago, IL 60614}
\maketitle
\begin{abstract}{We generalize the notion of trace identity to \jtr.  Our main result is that all \jtr\ identities of $M_{n,n}$ are consequence of those 
of degree $\frac12 n(n+3)$.  This also gives an indirect description of the queer trace identities of $M_n(E)$.}\end{abstract}
\keywords{Trace identities, \jtr-identities, queer trace identities, invariant theory, Sergeev algebra}

\section{Introduction}
\subsection{Traces, Supertraces and $J$-Traces}
Let $F$ be a field.  A trace function on an $F$-algebra $A$ is a linear map from $A$ to a commutative $F$-algebra, typically the center of~$A$, such
that $tr(ab)=tr(ba)$ for all $a,\ b\in A$.  Even if the 
image of the trace map is not in the center of~$A$, we
will always take it to be in a commutative algebra
that acts on~$A$.   One generalization that has been studied is that in
which $A$ is a $\bbz_2$-graded algebra, or superalgebra, $A=A_0\oplus A_1$.  In this case one 
studies supertraces, which are linear, homogeneous (degree preserving) maps from $A$ to a supercommutative ring which has a graded action on $A$
such that $str(ab)=(-1)^{\alpha\beta}str(ba)$  for all homogeneous $a\in A_\alpha$, $b\in A_\beta$.  By the homogenuity
of the supertrace we also have $a\cdot str(b)=(-1)^{\alpha\beta}str(b)a$. A basic example
is $A=M_n(E)$, where $E$ is the infinite Grassmann algebra with its usual grading making it into a supercommutative
algebra, and $A_i=M_n(E_i)$ for $i\in\bbz_2$.  If $a=(a_{ij})\in M_n(E)$ then the supertrace of~can be defined as 
$str(a)=\sum a_{ii}$. 

 There are other gradings and supertraces defined on $M_n(E)$.  Let $\omega$ be the involution on
$E$ defined by $\omega(e)=(-1)^\alpha e$ for all $e\in E_\alpha$.  Given $k+\ell=n$, we embed
$E$ into $M_n(E)$ by mapping each $e$ to
the  diagonal matrix
$$\omega_{k,\ell}(e)=diag(\underbrace{e,\ldots,e}_k,\underbrace{\omega(e),\ldots,\omega(e)}_\ell),$$
and define $M(k,\ell)_0$ to be the elements of $M_n(E)$ which commute with $E$, and $M(k,\ell)_1$ to
be the  ones that supercommute.  It is not hard to see that $M(k,\ell)_0$ consists of all $n\times n$ matrices of the form $\left(\begin{smallmatrix}
A&B\\C&D\end{smallmatrix}\right)$ in which $A\in M_k(E_0)$, $D\in M_\ell(E_0)$, and $B$ and~$C$ have entries in~$E_1$;
and $M(k,\ell)_1$ is the opposite.  The algebra $M(k,\ell)_0$ is generally denoted $M_{k,\ell}$.  The supertrace on $M(k,\ell)$
is defined on homogeneous elements of degree~$d$ as
$$str(a_{ij})=\sum_{i=1}^k a_{ii}-(-1)^d\sum_{i=k+1}^n a_{ii},$$ and extended to all of $M(k,\ell)$ by linearity.
It is useful to consider $M(k,\ell)$ to be an $E$-algebra by letting each $e\in E$ act as $\omega_{k,\ell}(e)$.
With this action we have that the supertrace is $E$-linear in the natural sense that
$str(ea)=e\cdot str(a)$ for all~$a$.

In this paper we will be interested in algebras~$A$ with supertrace with a degree one element~$j$ such $j^2=1$.  The principal example
is $M(n,n)$ and $j=\left(\begin{smallmatrix}0&I_n\\I_n&0\end{smallmatrix}\right)$.  For any such algebra $A_1=jA_0=A_0j$ and $A_0=jA_1=
A_1j$.  Moreover, if $a\in A_0$ commutes with~$j$ then
$$str(a)=str(j^2a)=-str(jaj)=-str(a)$$ and unless the characteristic is~2 this implies $str(a)=0$.
\subsection{Universal Algebras}
Let $X$ be an infinite set.  Define $\oF\langle X,tr\rangle$ to be the commutative $F$-algebra generated
by the symbols $tr(u)$ for $u$ in the free algebra $F\langle X\rangle$ with the relations $tr(\alpha u)=
\alpha tr(u)$, $tr(u+v)=tr(u)+tr(v)$ and $tr(uv)=tr(vu)$ for all $\alpha\in F$, $u,v\in F\langle X\rangle$.
There is an obvious trace map from $F\langle X\rangle$ to $\oF\langle X,tr\rangle$ with the universal 
property that given any algebra homomorphism $F\langle X\rangle\rightarrow A$, where $A$ is an
algebra with trace $A\rightarrow C$, there is a unique map $\oF\langle X,tr\rangle\rightarrow C$
making the following diagram commute:

$$\begin{CD} {F\langle X\rangle} @>>> A\\
@VVV @VVV\\
{\oF\langle X,tr\rangle} @>> >C
\end{CD}$$
Let $F\langle X,tr\rangle$ be the tensor product $\oF\langle X,tr\rangle\otimes F\langle X\rangle$.
There are two traditional approaches to defining a trace function on $F\langle X,tr\rangle$, the
issue being how to define $tr(tr(u))=tr(u)tr(1)$.  Either one defines a new commuting variable
$\gamma=tr(1)$, or else chooses an integer~$n$ and sets $tr(1)=n$ in order to study algebras 
with $tr(1)=n$.  At any rate, if $A$ is an $F$-algebra with trace function to the central subalgebra~$C$,
(and with $tr(1)=n$ if we are in the latter case), then the above universal property takes the simpler
form that given any set theoretic map $X\rightarrow A$ there is a unique extension to a  homomorphism $F\langle X,tr\rangle \rightarrow A$ which preserves trace.
  Elements of  $F\langle X,tr\rangle$ are called trace polynomials or mixed trace polynomials, and ones which are 
zero under every trace preserving homomorphism 
$F\langle X,tr\rangle\rightarrow A$, for a given~$A$, are called trace identities for~$A$; and
elements of  $\oF\langle X,tr\rangle$ are called pure trace polynomials and ones that are
identities for some~$A$ are called pure trace identities of~$A$.  The set of trace identities for~$A$
will be denoted $I(A,tr)$ and it is an ideal of $F\langle X,tr\rangle$. The quotient will be
denoted $U(A,tr)$ and it has the expected universal properties.  Likewise $\bar{I}(A,tr)=I(A,tr)\cap \oF
\langle X,tr,\rangle$ and $\bar{U}(A,tr)=\oF\langle X,tr\rangle/\bar{I}(A,tr) $.

Given two sets of variables $X$ and $Y$, considered
degree~0 and degree~1, respectively, there are easy generalizations of these concepts to $\oF\langle X,Y,
str\rangle$ and $F\langle X,Y,
str\rangle$, the free superalgebra with supertrace, including the notations of supertrace
identities or pure supertrace identities.  Note that $\oF\langle X,Y,
str\rangle$ is required to be supercentral in $F\langle X,Y,
str\rangle$ so that if $u\in \oF\langle X,Y,str\rangle$ is degree one, then it anticommutes with all degree one
elements.  For example, the Grassmann algebra has a supertrace map $str:E\rightarrow E$ given by the function $\omega$ defined above.  With respect to this
supertrace $E$ satisfies the mixed supertrace identities
$$str(x)=x\mbox{ and }str(y)=-y$$
and all supertrace identities of $E$ are consequences of these two.

For our purposes we will need  another generalization.  We will be considering
superalgebras with supertrace, which have a degree one element $j$ whose square is equal to~1.  In such a case we say that the
algebra has a $J$-trace.

 We define
$F\langle X,\jtr\rangle$ to be the free supertrace  algebra $F\langle X, \{J\},str\rangle$
modulo the relation that $J^2=1$.  Note that all variables are even and $J$ is odd.
  Elements of $F\langle X, \jtr\rangle$ are called $J$-trace polynomials, and elements of $\bar{F}
\langle X,\{J\},\jtr\rangle$ are called pure $J$-trace polynomials.   And, for
 algebras $A$ with supertrace and designated odd
 idempotent~$j$, one defines $I(A,\jtr)$ as the ideal
of \jtr\ identities for $A$ and $U(A,\jtr)$ the corresponding
universal algebra.  Likewise, $\bar{I}(A,\jtr)$ will be the 
pure \jtr\ identities, and $\bar{U}(A,\jtr)$ will be
the quotient $\oF\langle X,\jtr\rangle/\bar{I}(A,\jtr)$.  For example, $M(1,1)$ satisfies the mixed
\jtr\ identity 
\begin{equation}
2[Jx_1+x_1J,Jx_2+x_2J]=str(Jx_1+x_1J)str(Jx_2+x_2J)\label{eq:1}
\end{equation}

From one point of view \jtr\ identities are identities of $A_0$ since each
variable $x_i$ is degree zero and substituted by elements of~$A_0$ only.
For this reason we may speak of the \jtr\ of~$A_0$.
On the other hand, if we wanted to describe a graded \jtr\ identity
$f(x_1,\ldots,x_k,y_1,\ldots,y_\ell)$ in both even and odd variables
we could simply let $$g(x_1,\ldots,x_{k+\ell})=f(x_1,\ldots,x_k,Jx_{k+1},
\ldots,Jx_{k+\ell}).$$  Then $f$ will vanish under all graded substitutions
$x_i\mapsto a_i\in A_0$, $y_j\mapsto a_j\in A_1$ if and only if $g$
vanishes under all substitutions $x_i\mapsto a_i\in A_0$.
\subsection{\jtr-identities of $M_{n,n}$}
Our main result is the determination of the \jtr-identities of 
$M(n,n)$ in charcteristic zero.  Following
in the footsteps of Procesi, see~\cite{P}, our approach is based on invariant 
theory.  Procesi's work in~\cite{P} is based on the invariant theory
of the general linear group which in turn is based on the
double centralizer theorem between the general linear group
and the symmetric group.  The progression of ideas is similar
here.  In \cite{S85} Sergeev proved a double centralizer theorem for 
the queer superalgebra; and in \cite{B13} we used Sergeev's theory
to study the invariant theory of the queer superalgebra.  We proved
that the pure \jtr\ polynomials in~$d$ variables give all maps $M(n,n)^d\rightarrow
E$ invariant under the action of the queer superalgebra.  In the language of
invariant theory this theorem could be called the First Fundamental Theorem
for the queer superalgebra and in the same spirit the current work, giving
the relations between the \jtr\ polynomials would be the Second
Fundamental Theorem.

In section~2 of the current work we show how to associate
multilinear pure \jtr-polynomials with elements of the Sergeev
algebra and we show that this association has properties
similar to Razmyslov and Procesi's association of trace polynomials with
the group algebra of the symmetric group.  (For the reader
familiar with polynomial identities:  But not analogous to
Regev's association of ordinary multilinear  polynomials with
elements of the group algebra of the symmetric group.) Then in section~3
we describe which elements of the Sergeev algebra 
correspond to \jtr-polynomial identities for $M(n,n)$ and prove our main theorem, that
in characteristic zero all such identities are consequences
of the ones of degree $\binom{n+2}2$.  It then follows from the
non-degeneracy of the \jtr\ that all mixed \jtr-identities are
consequences of the ones of degree~$\binom{n+2}2-1=\frac12n(n+3)$.

\subsection{Queer Traces}
Readers familiar with Kemer's structure theory of p.i.~algebras (see~\cite{K}) will know that it is based on the three
families of verbally prime algebras, $M_n(F)$, $M_{k,\ell}$,
and $M_n(E)$.  Although the polynomial identities are
not well understood for any of the families, the trace
identities are known for the first two:  The case of trace identities of  matrices
was done by Razmysolv and Procesi, see~\cite{Ra74}
and~\cite{P}, and the case of $M_{k,\ell}$ by Razmyslov
and Berele, see \cite{Ra85} and \cite{B88}.  The third
family $M_n(E)$ is not graced with a natural trace, but it
does have a queer trace.  The queer trace of a matrix
with Grassmann entries is defined to be the degree one part
of the sum of the diagonal entries.  This map is similar
to a trace function in that it satisfies $qtr(ab)=qtr(ba)$, but
unlike  ordinary traces, queer trace anticommute, $qtr(a)
qtr(b)=-qtr(b)qtr(a).$  In this section we show how the
theory of \jtr-identities of $M(n,n)$ can be used to describe
the queer trace idenities of $M_{n}(E)$.

In general there are two ways to construct algebras with
queer trace from graded algebras with supertrace.  If
$A$ is such an algebra and if $str(A_0)=0$ then it is not
hard to see that $str(ab)=str(ba)$ and $str(a)str(b)=-str(b)str(a)$
and so the supertrace is a queer trace.  So, given any
superalgebra $A$ with supertrace, we can simply define
$qtr(a)$ to be the degree one part of $str(a)$.  This is how
the queer trace on $M_n(E)$ was defined.  The second construction
can be used if in addition
$A$ had an idempotent element $j\in A_1$.  In this case we could
let $B$ be the centralizer of $j$.  $B$ would be a graded
superalgebra with $str(B_0)=0$ and so the supertrace on~$B$
would be a queer trace.  For our purposes we will need
the case of $A=M(n,n)$.  The centralizer of $j$ is the set of matrices of the form $\left(\begin{smallmatrix} A&B\\ B&A
\end{smallmatrix}\right)$. This algebra is denoted $SM(n,n)$, and the degree zero part will be denoted $S\mnn$. ($S$ for symmetric.)

There are technical issues in defining mixed queer trace
identies since queer traces are not central.  Pure
queer trace polynomials can easily be defined as the
free algebra on the symbols $qtr(u)$ where $u$ is a monomial
modulo the obvious relations.   As an example, the Grassmann algebra satisfies the mixed queer trace identity
\begin{equation}[x_1,x_2]=2qtr(x_1)qtr(x_2)\label{eq:1a}\end{equation} and the pure queer trace identities
\begin{align}qtr(x_1[x_2,x_3])&=qtr(x_1)qtr(x_2)qtr(x_3),\label{eq:2}\\qtr(x_1)qtr(x_2x_3)&+
qtr(x_2)qtr(x_1x_3)+qtr(x_3)qtr(x_1x_2)=0.\label{eq:3}\end{align}
The reader may wish to show that \eqref{eq:2} and \eqref{eq:3} are each
consequences of \eqref{eq:1a}.  The intrepid reader may wish to show that all
queer trace identities (appropriately defined) of~$E$ are consequences of~\eqref{eq:1a}, and that all pure queer trace identities
are consequences of \eqref{eq:2} and \eqref{eq:3}.
  
There is an isomorphism from $S\mnn$ to $M_n(E)$ given by
$$\begin{pmatrix}A&B\\ B&A\end{pmatrix}\mapsto A+B.$$  It is
desirable to make this map queer trace preserving and to that end
we define $\widetilde{qtr}$ on $S\mnn$ as
$$\widetilde{qtr}\begin{pmatrix}A&B\\ B&A\end{pmatrix}=tr(B),$$
or, more naturally, as $\widetilde{qtr}(a)=\frac12str(ja)$,

Since we can describe all \jtr-identities of $M(n,n)$ the
following theorem will describe all queer trace identities of
$M_n(E)$.  If $A$ is \jtr-algebra then we make the
univeral algebra $U(A,\jtr)$ into an algebra with 
queer trace by modding out by the relation that
$J$ is central.  Denote by $Q$ the map $$Q:\oF 
\langle X,\jtr\rangle\rightarrow \oF\langle X,qtr\rangle,$$ from supertrace
polynomials to queer trace polynomials gotten by making
$J$ central and identifying $str(ja)$ with $2qtr(a)$.
Note that when we make $J$ central we make supertraces
of even terms~0.
\begin{thm} In characteristic zero, the ideal of queer trace identities of $M_n(E)$,
equals the ideal of \jtr-identities of $M(n,n)$ together with
the identity that $J$ is central.\end{thm}
\begin{proof} Since $M_n(E)$ is isomorphic to $S\mnn$\ as an algebra
with queer trace, we consider queer trace identities of $S\mnn$.  Also,
since $S\mnn$\ is the centralizer of $j$ in $\mnn$, all supertrace of
identities of $\mnn$\ become queer trace identities of $S\mnn$\ under
the map~$Q$.  Since all identities are consequences of multilinear ones,
 in order to prove the theorem it suffices
to prove that if $f(x_1,\ldots,x_k)$ is a multilinear queer trace identity
for $S\mnn$  there exists a \jtr\ polynomial $g(x_1,\ldots,x_k)$
such that $Q(g)=f$.
First, let $f_1(x_1,\ldots,x_k)$ be any supertrace polynomial such that
$Q(f_1)=f$.  Then, $f_1$ will be a \jtr-identity for $S\mnn$, and so 
we let $$g(x_1,\ldots,x_k)=f_1(x_1+Jx_1J,\ldots,x_k+Jx_kJ).$$ For
each degree zero element~$a$ of $\mnn$ $a+jaj$ commutes with~$j$ and
so is a degree zero element of $S\mnn$.  Hence, $g$ is a \jtr-identity
for \mnn\ and $Q(g)=f(2x_1,\ldots,2x_k)=2^kf(x_1,\ldots,x_k)$ for
all $x_1,\ldots,x_k\in S\mnn$.
\end{proof}
\section{Properties of Multilinear $J$-Trace Identities}
\subsection{The Sergeev algebra}
Given a positive integer $d$ the Clifford algebra $C(d)$ is the $2^d$-dimensional algebra generated by the
anticommuting idempotents $c_1,\ldots,c_d$.  The symmetric group $S_d$ acts on
$C(d)$, and the Segeev algebra $W(d)$ is the semidirect product of $FS_d$ with $C(d)$.
More concretely, $S(d)$ is spanned by all $\sigma c_1^{\ep_1}\cdots c_d^{\ep_d}$  with
 relations $$c_i^2=1,\ c_ic_j=-c_jc_i,\ \sigma c_i=c_{\sigma(i)}\sigma.$$  In keeping with \cite{B13} we will be using the graded opposite of $W(d)$ instead
 of $W(d)$.  It is generated by permutations $\sigma$ and by Clifford elements $C_i$
 with relations 
 \begin{equation}C_i^2=1,\ C_iC_j=-C_jC_i, C_i\sigma=\sigma C_{\sigma(i)}, \sigma.\tau=\tau\sigma
 \label{eq:4}\end{equation}
 for $i,j=1,\ldots,d$ and $\sigma,\tau\in S_d$.
 In the last relation the dot is for the product in $W(d)^o$ and the juxtaposition
 $\tau\sigma$ indicates the ordinary product in $S_d$.
 
There is a well-known identification of elements of $FS_d$ with multilinear, degree~$d$
pure trace polynomials.  If $\sigma$ has cycle decomposition $$\sigma=(i_1,\ldots,i_a)\cdots
(j_1,\ldots,j_b)$$ then $tr_\sigma(x_1,\ldots,x_d)$ is defined to be
\begin{equation}tr(x_{i_1}\cdots x_{i_a})\cdots tr(x_{j_1}\cdots x_{j_b}).\label{eq:4.1}\end{equation}  Note that this is well-defined
because cycles commute and traces commute, and elements can be permuted cyclically
in both traces and cycles.  Another nice property of this 
identification is that
\begin{equation}tr_{\sigma\tau\sigma^{-1}}(x_1,\ldots,
x_d)=tr_\tau(x_{\sigma(1)},\ldots,x_{\sigma(d)})
\label{eq:5}\end{equation}  

Hypothetically, one would like to identify elements of the
Sergeev algebra with multilinear $J$-trace polynomials via
$$\sigma C_1^{\ep_1}\cdots C_d^{\ep_d}\rightarrow tr_\sigma(J^{\ep_1}x_1,\ldots,J^{\ep_d}
x_d).$$  Unfortunately, such an identification would not be well-defined.  Two pure
$J$-traces may commute or anticommute:  They will anticommute if they are each of odd
degree in~$J$.  Likewise, permuting the elements within a trace may result in a negative sign.  On the other hand, for given $\sigma c\in W(d)$ there are only two different $\jtr$
polynomials that could result from the above putative identification and they
are negatives of each other, and so we need to choose one in a natural way.

To identify elements of the Sergeev algebra with $\jtr$~polynomials, we first define $str_\sigma$ for
 permutations. Assume that some of the variables $x_1,
\ldots,x_d$ are even and the remainder odd, and let
$e_1,\ldots,e_d$ be supercentral variables with $\deg x_i=
\deg e_i$ for all~$i$, and such that $str(e_iu)=e_istr(u)$ for
all~$u$.  Then $str_\sigma(x_1,\ldots,x_d)$ is
uniquely determined by the condition
$$tr_\sigma(e_1x_1,\ldots,e_dx_d)=e_d\cdots e_1 str_\sigma
(x_1,\ldots,x_d).$$  If $w=\sigma C_1^{\ep_1}\ldots C_d^{\ep_d}
\in W(d)^o$ we define
$$jtr_w(x_1,\ldots,x_d)=str_\sigma(J^{\ep_1}x_1,\ldots, J^{\ep_d}
x_d),$$  recalling that the $x_i$ are all even and $J$ is odd.
Combining the previous two equations we get
\begin{equation}e_d\cdots e_1 jtr_w(x_1,\ldots,x_d)
=tr_\sigma(e_1 J^{\ep_1}x_1,\ldots,e_d J^{\ep_d}x_d),
\label{eq:j}\end{equation}
which can be taken as the definition of $jtr_w$ when $w$ is a monomial.
Finally, using linearity we can define
$jtr_w$ for any $w\in W(d)^o$.

More generally, if 
$w=\sigma C_{i_1}\cdots C_{i_k}$ with the $i_a$ distinct, then $jtr_w$ is gotten
from $tr_\sigma(x_1,\ldots,x_d)$ by substituting $Jx_{i_a}$
for each $x_{i_a}$ and pulling out a factor of $e_{i_k}\cdots e_{i_1}$ on the left, whether the sequence of $i_a$ is
increasing or not. 

\subsection{Conjugation and \jtr\ Polynomials}
The goal of this section is to generalize \eqref{eq:5} to $\jtr$
polynomials.  There are two generalizations, one deals with conjugation
by a permutation and the other deals with conjugation 
by a generator of the Clifford algebra.
\begin{lem} Let $w$  be an element of~$W(d)^o$ and let $\sigma\in S_d\subseteq
W(d)^o$ be
a permutation.  Then $jtr_{\sigma w\sigma^{-1}}(x_1,
\ldots,x_d)=jtr_w(x_{\sigma^{-1}(1)},\ldots,x_{\sigma^{-1}(d)}).$\label{lem:2}\end{lem}
As the proof will show, there is a $\sigma^{-1}$ instead of a $\sigma$
in the formula unlike \eqref{eq:5} because we are using the opposite of
$S_d$.
\begin{proof} Without loss, we may take $w=\tau
 C_1^{\ep_1}\cdots C_d^{\ep_d}=\tau c$ and so by~\eqref{eq:4}
$$\sigma w\sigma^{-1}=\sigma.\tau.\sigma^{-1}C_{\sigma^{-1}(1)}^{\ep_1}
\cdots C_{\sigma^{-1}(d)}^{\ep_d}=\sigma^{-1}\tau\sigma C_{\sigma^{-1}(1)}^{\ep_1}
\cdots C_{\sigma^{-1}(d)}^{\ep_d}.$$  Then 
$$ C_{\sigma^{-1}(1)}^{\ep_1}
\cdots C_{\sigma^{-1}(d)}^{\ep_d} =
gC_1^{\ep_{\sigma(1)}}\cdots C_d^{\ep_{\sigma(d)}},$$ where
$g=\pm1$. Let $e_i$ be 
supercentral with degrees equal to $\ep_{\sigma(i)}$.  Then
$$ e_d\ldots e_1 jtr_{\sigma w\sigma^{-1}}(x_1,\ldots,x_d)
=g\cdot tr_{\sigma^{-1}\tau\sigma}(e_1J^{\ep_{\sigma(1)}}x_1,\ldots,
e_dJ^{\ep_{\sigma(d)}}x_d).$$
By \eqref{eq:5} this equals
\begin{align*}g\cdot tr_\tau(e_{\sigma^{-1}(1)}&J^{\ep_1}x_{\sigma^{-1}(1)},
\ldots,e_{\sigma^{-1}(d)}J^{\ep_d}x_{\sigma^{-1}(d)})\\ &=
g\cdot e_{\sigma^{-1}(d)}\ldots e_{\sigma^{-1}(1)}jtr_{wc}(x_{\sigma^{-1}(1)},\ldots,
x_{\sigma^{-1}(d)})\\
&=\pm e_d\ldots e_1 jtr_{wc}(x_{\sigma^{-1}(1)},\ldots,
x_{\sigma^{-1}(d)}).
\end{align*}
To complete the proof we need to check that the sign is positive.  Since the
degree of $e_i$ equals the degree of $\ep_{\sigma(i)}$, the definition
of~$g$ implies that $$ge_1\ldots e_d=e_{\sigma^{-1}(1)}\ldots e_{\sigma^{-1}(d)}.$$  
This is not exactly what we want because the indices are increasing rather
than decreasing.  We separate the last step of the proof into a sublemma.
\begin{subl*} Let $e_1,\ldots,e_d$ be homogeneous elements of the Grassmann
algebra~$E$ and let $\sigma\in S_d$ be a permutation such that
$ge_1\ldots e_d=e_{\sigma^{-1}(1)}\ldots e_{\sigma^{-1}(d)}.$  Then
$ge_d\ldots e_1=e_{\sigma^{-1}(d)}\ldots e_{\sigma^{-1}(1)}.$\end{subl*}
\begin{proof}[Proof of Sublemma:] Let $E$ be the Grassmann algebra on the
vector space $V$ with basis $v_1,v_2,\ldots$ and assume without loss of generality
that each $e_i$ is a monomial in the $v_j$.  There is a  antiisomorphism
$\theta$ on $E$ given by $$v_{i_1}\cdots v_{i_t}\mapsto v_{i_t}\cdots v_{i_1}.$$
For each monomial $e_i$ we have $\theta(e_i)=g_ie_i$ where $g_i=\pm1$.
Now,
\begin{align*}
\theta(e_d\cdots e_1)&=\theta(e_1)\cdots\theta(e_d)\\
&=g_1\cdots g_d e_1\cdots e_d\\
&=gg_1\cdots g_d e_{\sigma^{-1}(1)}\ldots e_{\sigma^{-1}(d)}\\
&=g\theta(e_{\sigma^{-1}(1)})\ldots \theta(e_{\sigma^{-1}(d)})\\
&=g\theta(e_{\sigma^{-1}(d)}\ldots e_{\sigma^{-1}(1)}).
\end{align*}
Since $\theta$ is a one-to-one linear transformation the sublemma, and hence
the lemma, follow.
\end{proof}\noqed
\end{proof}
\begin{lem} Let $w=\sg c =\sg C_1^{\ep_1}\cdots C_d^{\ep_d}\in W(d)^o$ be a monomial.  Then $$jtr_{C_iwC_i}(x_1,\ldots,
x_d)=\pm jtr_w(x_1,\ldots,Jx_iJ,\ldots,x_d),$$ where the sign is $(-1)^{\sum \ep_j}$.\label{lem:3}\end{lem}
\begin{proof}   There are  three slightly different cases
depending on whether $\sigma(i)$ is less than, equal to, or
greater than~$i$.  Trusting the reader to do the other two
 cases, we assume that $\sigma(i)<i$.  In this case
$$C_i\sigma cC_i=\sigma C_{\sigma(i)}c C_i=g_1\sigma c',$$
where
$$c'=C_1^{\ep_1}\cdots C_{\sigma(i)}^{\ep_{1+\sigma(i)}}\cdots
C_i^{1+\ep_i}\cdots C_d^{\ep_d}$$ and where
$g_1=(-1)^{\gamma_1}$ and $$\gamma_1=\ep_1+\cdots+\ep_{\sg(i)-1}+\ep_{i+1}+\cdots
\ep_d.$$  To compute $jtr_{\sigma
c'}$ we choose $e_j$ with degrees $\ep_j$, unless $j=i$ or~$\sigma(i)$ in which case the degree is one greater.  Now
\begin{equation}e_d\cdots e_1 jtr_{\sigma c'}(x_1,\ldots,x_d)=
tr_\sigma(e_1J^{\ep_1}x_1,\ldots,e_dJ^{\ep_d}x_d).
\label{eq:6.5}\end{equation}
The $i$-th term is $C_iJ^{1+\ep_i}x_i$ and the $\sigma(i)$ term is 
$e_{\sigma(i)}J^{1+\ep_{\sigma(i)}}x_{\sigma(i)}$.  We can 
assume without loss of generality that $e_{\sigma(i)}$ factors as $f_1f_2$, where $f_1$
is of degree~1 and $f_2$ is of degree~$\ep_{\sigma(i)}$.  An
important but elementary fact about trace monomials is 
that if $y_{\sigma(i)}=ab$ then
$tr_\sg(y_1,\ldots,y_d)=tr_\sg(y_1,\ldots,y_ia,\ldots,b,\ldots
y_d),$ where in the right hand side of the equation $y_i$
is replaced by $y_ia$ and $y_{\sg(i)}$ is replaced by~$b$.
In order to apply this to \eqref{eq:6.5} we first note
that $y_{\sg(i)}=f_1f_2J^{1+\ep_{\sg(i)}}x_{\sg(i)}$ which we re-write
as $(-1)^{\ep_{\sigma(i)}}f_1Jf_2J^{\ep_{\sg(i)}}x_{\sg(i)}$.
For the sake of bookkeeping, we let $g_2=(-1)^{\gamma_2}$
where $\gamma_2=\ep_{\sg(i)}$.  Then~\eqref{eq:6.5}
equals
$$g_2\cdot tr_\sg(e_1J^{\ep_1}x_1,\ldots,f_2J^{\ep_{\sg(i)}}x_{\sg(i)},\ldots,e_iJ^{1+\ep_i}x_i
f_1J,\ldots).$$
Since $f_1$ is degree one and supercentral, the $i^{th}$ argument equals $f_1e_iJ^{\ep_i}Jx_iJ$, and 
since $f_1e_i$ has degree $\ep_i$ this equals
$$g_2\cdot e_d\cdots e_{i+1} f_1e_i e_{i-1}\cdots
 e_{\sg(i)+1} f_2e_{\sg(i)-1}\cdots  e_1 jtr_w(x_1,
\ldots,Jx_iJ,\ldots,x_n)$$ and to complete the proof
we need only pull the $f_1$ back to be in front of the
$f_2$ and compute the sign.  First, 
$$e_d\cdots f_1\cdots f_2e_i\cdots e_1=g_3\cdot
e_d\cdots e_1$$ where $g_3=(-1)^{\gamma_3}$ and
$\gamma_3=\ep_{\sg(i)+1}\cdots+\ep_{i}$.  Altogether
the sign is $g_1g_2g_3=(-1)^{\gamma_1+\gamma_2+
\gamma_3}$ and $\gamma_1+\gamma_2+\gamma_3=
\ep_1+\cdots+\ep_d$,which proves the lemma.
\end{proof}
\subsection{Multilinear Identities}
Given a monomial $c=C_1^{\ep_1}\cdots C_d^{\ep_d}$
in $C(d)^o$ we define $|c|$ to be $\sum \ep_j$ in $\bbz_2$,
and we sometimes refer to $c$ as even or odd depending
on whether $|c|$ is~0 or~1.  Likewise, for 
$\sg c\in W(d)^o$ we define $|\sg c|$ to be $|c|$.
\begin{lem} Let $A$ be an algebra with a \jtr.  Assume that $$u=\sum_{\sigma,\ep} \alpha(\sigma,\ep)\sigma c^\ep\in W(d)$$
is such that $jtr_u$ identity of~$A$, and write $u$ as $u=u_0
+u_1$, the even terms plus the odd terms.  Then $jtr_{u_0}$
and $jtr_{u_1}$ are each identities for~$A$ .\label{lem:4}
\end{lem}
\begin{proof} Given any $a_1,\ldots,a_d\in A_0$, $$jtr_u(a_1,\ldots,a_d)=jtr_{u_0}(a_1,\ldots,a_d)
+jtr_{u_1}(a_1,\ldots,a_d)=0.$$ But since the first term is in ~$A_0$ and the second in~$A_1$
each must be zero.\end{proof}
Combining this lemma with lemmas~\ref{lem:2} and ~\ref{lem:3} of the previous
section we get this theorem.
\begin{thm} Let $A$ be a \jtr\ algebra and let $jtr_f$ be an
identity for $A$ for some $f\in W(d)^o$.  Let $w=\sg c$ be
a monomial in $W(d)^o$.  Then $jtr_{wfw^{-1}}$ is also an
identity for~$A$.\label{cor:1}\qed
\end{thm}\qed

The correspondence $w\leftrightarrow jtr_w$ is a vector space isomorphism between
$W(d)^o$ and the degree~$d$, multilinear, pure \jtr\ polynomials in $x_1,\ldots,x_d$, which
we denote $JTR_d$.  Using this identification we may consider $JTR_d$ as a bimodule for
$W(d)^o$.  Also note that the incusions $W(d)^o\subseteq W(e)^o$ for $d\le e$
correspond to the inclusions $JTR_d\subseteq JTR_{d+1}$ gotten by correspnding 
$f(\xd)$ to $f(\xd)str(x_{d+1})\cdots str(x_e)$.  For a fixed \jtr-algebra~$A$ we let $I_d\subseteq JTR_d$ be the set of identities
for~$A$ in~$JTR_d$.  
\begin{conv*}Theorem~\ref{cor:1} shows that $I_d$ is invariant under conjugation
by monomials in $W(d)^o$ but in general it need not be invariant under right or left
multiplication.  For the remainder of  this section we will consider the special case in which $I_d$ is both
a left and right $W(d)^o$ module and we will prove that the elements of  $W(e) I_d W(e)$
are all consequences of $I_d$ and hence are in $I_e$, 
for all $d\le e$.\end{conv*}
  We remark that the corresponding statement is a key lemma in the derivations
of the trace identities for matrices and for $M_{k,\ell}$ by Razmyslov, Procesi and Berele mentioned
previously.

First of all, by induction, it suffices to prove the statement for $e=d+1$.  Next, it suffices
to prove that if $f\in I_d$ then $fw,\ wf\in I_{d+1}$ for each monomial $w\in W(d+1)^o$,
and by Theorem~\ref{cor:1} we need only consider the case of $wf$.  

\begin{lem}Let $\sigma\in S_{d+1}$ fix $d+1$. Then $tr_{\sigma(i,d+1)}(x_1,\ldots,x_{d+1})
=tr_\sigma(x_1,\ldots,x_{d+1}x_i,\ldots,x_d)$.\label{lem:5}
\end{lem}
\begin{proof}  The proof follows from the computation
$$(i,a,\ldots,b)(i,d+1)=(i,d+1,a,\ldots,b)$$
together with equation \eqref{eq:4.1}.\end{proof}

\begin{lem}
Let $f=jtr_w$ for some $w\in I_d$.  Then for each $u\in
W(d+1)^o$, $uf\equiv jtr_{uw}$ is a consequence of~$I_n$.\label{lem:2.5}
\end{lem}
\begin{proof}
By the previous lemma we may assume without loss of generality
that either all terms of $w$ are even or all terms
are odd.

 Let $u=\sigma c$, $\sigma\in S_{d+1}$, $c\in C(d+1)$.
We first consider the case of $\sigma\in S_d$ so that $\sigma(d+1)=d+1$.  In this case
$u$ can be re-written as $C_{d+1}^{\ep_{d+1}}u'$ where $u'\in C(d)$ and so $u'f\in
I_d$.  Hence, in this case, we may replace $f$ by $u'f$ and
so assume without loss that $u=C_{d+1}^{\ep_{d+1}}$.  Also,
by Lemma~\ref{lem:4} we may assume that every term in $w$ is of degree~$g$ mod~2.
In this case
\begin{align*}
jtr_{C_{d+1}^{\ep_{d+1}}w}(x_1,\ldots,x_{d+1})&=(-1)^{g\ep_{d+1}} jtr_{wC_{d+1}^{\ep_{d+1}}}(x_1,
\ldots,x_{d+1})\\ &=(-1)^{g\ep_{d+1}}jtr_{w}(x_1,\ldots,x_d)str(J^{\ep_{d+1}}x_{d+1}),
\end{align*}
which is certainly a consequence of $f$.

If $w=\sigma c$ with $\sigma\notin S_d$ we 
may write $\sigma$ as $(i,d+1)\sigma'$, where $\sigma'(d+1)=d+1$;
 and we  write $c$ as $c=C_{d+1}^{\ep_{d+1}}c'$ where $c'\in C(d)$.  Then $C_{d+1}$ commutes
with $\sigma'$ and so  $w=(i,d+1)C_{d+1}^{\ep_{d+1}}w'$ where $w'\in W(d)^o$, and since
$W(d)^oI_n\subseteq I_n$ it suffices to consider
the case in which $w=(i,d+1)C_{d+1}^{\ep_{d+1}}.$
We now compute the \jtr\ polynomial corresponding to $wu$ where  $w=(i,d+1)C_{d+1}^{\ep_{d+1}}$ and~$u$
is a monomial $u=\sigma c=\sigma C_1^{\ep_1}\ldots C_d^{\ep_d}$.
First note that $wu=(-1)^\gamma (i,d+1).uc_{d+1}^{\ep_{d+1}}$, where $\gamma$ equals 
$\ep_{d+1}(\ep_1+\ldots+\ep_d)$; and since the multiplication of the permutations in $W(d+1)^o$ is
in the opposite their product in $S_{d+1}$, $wu$ equals $(-1)^\gamma\sigma(i,d+1)cC_{d+1}^{\ep_{d+1}}.$  Let $e_i$ have degree $\ep_i$.  Then by \eqref{eq:j}
\begin{equation}e_{d+1}\ldots e_{1}jtr_{wu}=(-1)^\gamma \cdot tr_{\sigma(i,d+1)}(e_1J^{\ep_1}x_1,\ldots,e_{d+1}J^{\ep_{d+1}}
x_{d+1}).\label{eq:7}\end{equation} Applying Lemma~\ref{lem:5},  \eqref{eq:7} equals
$$
(-1)^\gamma\cdot tr_\sigma(e_1J^{\ep_1}x_1,\ldots,e_iJ^{\ep_i}x_ie_{d+1}J^{\ep_{d+1}}x_{d+1},\ldots,
e_{d}J^{\ep_{d}})
x_{d}.$$  Focusing on the $i^{th}$ term:
$$e_iJ^{\ep_i}x_ie_{d+1}J^{\ep_{d+1}}x_{d+1}=e_{d+1}e_iJ^{\ep_i+\ep_{d+1}}J^{\ep_{d+1}}x_{d+1}J^{\ep_{d+1}}
x_{d+1}$$ and
so \eqref{eq:7} equals  $$
(-1)^\gamma\cdot e_{d}\ldots e_{i+1}e_{d+1}e_{i}\ldots e_1 jtr_{\sigma c'}(x_1,\ldots, J^{\ep_{d+1}}x_{i}J^{\ep_{d+1}}x_{d+1},\ldots,x_d),$$  where $c'=C_1^{\ep_1}\cdots C_i^{\ep_i+\ep_{d+1}}\cdots C_d^{\ep_d}$.  The above
equals 
$$(-1)^\gamma\cdot e_{d+1}\ldots e_{1} jtr_{u C_i^{\ep_{d+1}}}(x_1,\ldots, J^{\ep_{d+1}}x_{i}J^{\ep_{d+1}}x_{d+1},\ldots,x_d).$$  and so
$$jtr_{wu}=(-1)^\gamma\cdot  jtr_{u C_i^{\ep_{d+1}}}(x_1,\ldots, J^{\ep_{d+1}}x_{i}J^{\ep_{d+1}}x_{d+1},\ldots,x_d).$$
Since $u C_i^{\ep_{d+1}}\in W(d)^o$ and the above polynomial
 is gotten from  $jtr_{u C_i^{\ep_{d+1}}}$ by the
 homogeneous substitution $x_i\mapsto
  J^{\ep_{d+1}}x_{i}J^{\ep_{d+1}}x_{d+1}$ it is
a consequence of it.
\end{proof}
\begin{thm} Assume that $W(d)^oI_dW(d)^o=I_d$.  Then for all
$e\ge d$ the elements of $W(e)^oI_dW(e)^o$ are all consequences
of $I_d$.  In particular, they are all identities of~$A$.\label{thm:2.7}
\end{thm}
\begin{proof} Follows from the previous lemma by induction on $e$.
\end{proof}
\section{\jtr-Identies of $M(n,n)$}
\subsection{Theorems of Sergeev and Berele}
For the remainder of the paper we take $F$ to be characteristic zero.
Let $U$ be a $\bbz_2$-graded vector space with graded dimension $(n,n)$ and assume that $J$ is a degree~1 idempotent isomorphism of
$U$, i.e., a linear transformation with $J^2=id$, $J(U_0)=U_1$ and $J(U_1)=U_0$.  In\cite{S85} Sergeev defined an action of $W(d)^o$ on $U^{\otimes d}$
such that if $u_1,\ldots,u_d\in U$ are homogeneous, then 
$$\sigma(u_1\otimes\cdots \otimes u_d)=\pm u_{\sigma(1)}\otimes\cdots\otimes u_{\sigma(d)}$$
and $$C_i(u_1\otimes \cdots\otimes u_d)=\pm u_1\otimes \cdots\otimes Ju_i\otimes\cdots\otimes u_d,$$
where the sign depends on how many degree one elements are switched.   The exact details are not necessary for our
purposes.  Sergeev proved two basic properties of this action, one about the image of $W(d)^o$ and one
about the kernel of the action.  The former result is that the image of $W(d)^o$ in $End(U^{\otimes d})$
equals the centralizer of the action of $q(n)$, the queer superalgebra.  For the kernel, it is important to know
that $W(d)^o$ is a direct sum of two sided ideals $\sum\oplus I_\lambda $, each $I_\lambda$ is a simple graded algebra,
and the subscripts are naturally indexed by partitions of $d$ into distinct parts.  If we write
$\phi:W(d)^o\rightarrow End(U^{\otimes d})$ then Sergeev proved that the kernel is the sum $\sum\oplus I_\lambda$ summed over all
partitions of height greater than~$n$.  Note that the smallest such partition would be $\delta=(n+1,n,\ldots,1)$ and in
general if $\lambda$ is a partition with distinct parts, $\lambda$ will have height greater than~$n$ if and only if
$\delta\subseteq\lambda$.  For convenience, we record these results of Sergeev.

\begin{thm}  The algebra $W(d)^o$ is semisimple as an algebra and as a graded algebra.  As a graded algebra it
decomposes as a direct sum of graded simple ideals $\sum\oplus I_\lambda$ indexed by partitions of~$d$ with all
distinct parts.  There is a natural map $\phi:W(d)^o\rightarrow End(\vd)$ whose kernel is the sum of the $I_\lambda$
where $\lambda$ has height greater than~$d$.\label{thm:Ser}\end{thm}

\begin{defn}We let $DP(d)$ be the set if partitions of $d$ into distinct parts.\end{defn}

In \cite{B13} we let $V$ be a $2n$ dimensional space over the Grassmann algebra~$E$. 
The algebra $End_E(V)$ is isomorphic to $M(n,n)$ as a superalgebra with supertrace. Using the notion of extending an $F$-linear map to an $E$ linear one, we  
embed $End(U^{\otimes d})$ into $End(V^{\otimes d})$.  Composing this with
$\phi$ creates a map
$$\Phi: W(d)^o\rightarrow End(V^{\otimes d})$$
which has the same kernel as $\phi$.  Next, adapting some standard techniques from
invariant theory we constructed an isomorphism
$$End(V^{\otimes d})\cong (End(V)^{\otimes d})^*.$$  We denoted by $T_w$ the functional
on $End(V)^{\otimes d})$ corresponding to $w\in W(d)^o$.  

In section 5 of \cite{B13} we computed $T_w(A_1\ootimes A_d)$ for
$A_1,\ldots,A_d$ degree zero elements of $End(V)$. Let $w=\sigma C_1^{\ep_1}\cdots C_d^{\ep_d}\in W(d)^o$.  In lemma~5.8 we showed that 
$$T_w(A_1\ootimes A_d)=T_\sigma(\tilde{P}^{\ep_1}A_1\ootimes\tilde{P}^{\ep_d}A_d),$$
where $\tilde{P}$ is a degree one element of $End(V)$ corresponding to~$J$.  Letting $e_i$ be Grassmann elements of degree $\ep_i$ we
showed in the proof of lemma~5.7 that
$$e_d\cdots e_1 T_\sigma(\tilde{P}^{\ep_1}A_1\ootimes\tilde{P}^{\ep_d}A_d)=
T_\sigma(e_1\tilde{P}^{\ep_1}A_1\ootimes e_d\tilde{P}^{\ep_d}).$$
Continuing to work backwards through the lemmas of
\cite{B13}, lemma~5.6 can be used to evaluate the right hand side of
this equation, since the factors in the tensor product are now all degree zero.  It implies
that $$T_\sigma(e_1\tilde{P}^{\ep_1}A_1\ootimes e_d\tilde{P}^{\ep_d}A_d)=
tr_\sigma(e_1\tilde{P}^{\ep_1}A_1,\ldots, e_d\tilde{P}^{\ep_d}A_d).$$  Combining all this
we get
$$e_d\ldots e_1T_w(A_1\ootimes A_d)=tr_\sigma(e_1\tilde{P}^{\ep_1}A_1,\ldots, e_d\tilde{P}^{\ep_d}A_d),$$ and comparing with \eqref{eq:j} gives this lemma:

\begin{lem} If  $A_1,\ldots,A_d\in End(V)_0$ and $w\in W(d)^o$ then
$T_w(A_1\ootimes A_d)=jtr_w(A_1,\ldots,A_d)$.\qed
\end{lem}
It follows that $w\in W(d)^o$ corresponds to a \jtr-identity
for $End(V)$ if and only if $T_w$ is zero.  But $T_w=0$
precisely when $w$ is in the kernel of $\Phi$, which is the
same as the kernal of $\phi$ and which was identified by
Sergeev.  Finally, $End(V)_0$ is isomorphic to $M(n,n)$ as an
algebra with \jtr.  We now have the main theorem of this section:
\begin{thm} Let $w\in W(d)^o$.  Then $jtr_w$ is a \jtr-identity
for $M(n,n)$ if and only if $w$ lies in the two sided ideal
$\sum\oplus I_\lambda$, summed over all $\lambda$ of 
height greater than~$n$.\label{thm:3.2}\qed
\end{thm}
\subsection{A Combinatorial Lemma}
Theorem~\ref{thm:3.2} tells us that the multilinear,
degree~$d$ identities of $M(n,n)$, $I_d=I_d(M(n,n))$, form a two sided ideal in each $W(d)^o$; and 
Theorem~\ref{thm:2.7} says that the two sided ideal induced
from $I_d$, considered as a subspace of $W(d)^o,$ to each $W(e)^o$, $e>d$ consists of algebraic consequences of $I_d$.
Fortunately, the question of how ideals in $W(d)^o$ induce up
is easily handled by a theorem of Schur and
Jozefiak, and a theorem of Stembridge, as we now relate.  We first focus on the case of $W(d+1)^o I_\lm W(d+1)^o$, where $\lm$ is a partition
of~$d$ in $DP(d)$ so $I_\lm$ is an ideal of $W(d)^o$.
\begin{lem} If $I\triangleleft W(d)^o$ is a two sided ideal, and if $W(d+1)^o I$ decomposes as a direct sum of left ideals as $\sum m_\mu  M_\mu$,
where $M_\mu\subseteq I_\mu$ is a minimal left ideal and $m_\mu\ne0$,
 then $W(d+1)^o IW(d+1)^o=\sum I_\mu$ summed over all $\mu$ with $m_\mu\ne0$.\end{lem}
\begin{proof} Since $M_\mu$ is a right ideal contained in the
minimal two-sided ideal $I_\mu$, $$M_\mu I_\nu=\begin{cases} 0,&\mu\ne\mu\\ I_\mu,&\mu=
\nu
\end{cases}$$
\end{proof}
The computation of $W(d+1)^o I_\lm W(d+1)^o$ is accomplished in the next lemma.
\begin{lem} Let $\lm\in DP(d)$.  Then $W(d+1)^oI_\lm W(d+1)^o=\sum\oplus I_\mu$ where $\mu$ runs  over all partitions in $DP(d+1)$ containing~$\lm.$
\end{lem}
\begin{proof}As a right $W(d)^o$ module $I_\lm$ decomposes
as a direct sum of isomorphic simple modules $M_\lm$.
Now $$W(d+1)^o M_\lm\cong W(d+1)\otimes_{W(d)^o}M_\lm
\cong W(d+1)^o/W(d)^o\otimes_F M_\lm.$$
We compare this to 
$$M_{[1]}\hat\otimes M_\lm=W(d+1)^o\otimes_{W(1)^o\times
W(d)^o}\left(M_{[1]}\otimes_F M_\lm\right).$$  However,
$W(1)^o$ is graded simple of dimension $(1,1)$ and  $M_{[1]}$ is isomorphic to $W(1)^o$.  It follows that
the right hand side is isomorphic to $$W(d+1)^o\otimes_{W(d)^o} M_\lm$$ and so $W(d+1)^o M_\lm$
is isomorphic to $M_{[1]}\hat\otimes M_\lm$.  The outer
tensor product now follows from the theorems we alluded
to earlier.  First, just like the outer tensor product of 
modules for the symmetric group is reflected in the product of Schur functions, Schur and Jozefiak proved 
that the outer product of Sergeev modules is reflected
in the product of Q-Schur functions.  See \cite{M} for
an account of these functions.
\begin{thm*} [Schur, Jozefiak]The Grothendick ring of $W(d)^o$ modules is
isomorphic to the ring of Q-Schur functions with the isomorphism
given by $M_\lm\mapsto Q_\lm(x)$.
\end{thm*}
This means that $M_\mu\hat\otimes M_\nu=\sum m(\lm;\mu,nu)
M_\lm$ if and only if $Q_\mu(x)Q_\nu(x)=\sum m(\lm;\mu,\nu)Q_\lm(x).$  The latter products were computed  by Stembridge in~\cite{Stem}, see also \cite{M}~(8.18).   Stembridge
proved an analogue of the Littlewood-Richardson rule for
the functions $P_\lm(x)$ which have the property that each $Q_\lm(x)$
is a (non-zero) constant multiple of $P_\lm(x)$ and so
we can use Stembridge's rule to determine when the coefficients we are calling $m(\lm;\mu,\nu)$ are non-zero.  We will need only the following special case:
\begin{thm*}[Stembridge] Let $\lm\in DP(n)$ and $\mu\in DP(n+1)$.  Then
$m(\mu;\lm,[1])\ne0$ if and only if $\lm\subseteq\mu$.
\end{thm*}
Our lemma now follows.
\end{proof}
The following corollary follows by induction.
\begin{cor}Let $\delta$ be the partition $(n+1,n,\ldots,1)\in DP({n+2\choose2})$ and let $e>{n+2\choose2}$.  Then
$W(e)^o I_\delta W(e)^o$ equals the sum $\sum\oplus I_\lm$,
where $\lm$ runs over partitions of $e$ with distinct parts, of height
greater than~$n$.\label{cor:3.9}
\end{cor}
\subsection*{The Main Theorem}
Theorem~\ref{thm:2.7}, Theorem~\ref{thm:3.2} and Corollary~\ref{cor:3.9} are the ingredients we need to prove our
main theorem.
\begin{thm} All pure \jtr\ identities of $M(n,n)$ are consequences of
the identities of degree $n+2\choose2$.
\end{thm}
\begin{proof} By theorem~\ref{thm:3.2} there are no \jtr~identities
of degree less than $d={n+2\choose2}$, and in degree $n+2\choose 2$
 the multilinear identities correspond to the two sided ideal of $W(d)^o$,  
 $I_\delta$.  By theorem~\ref{thm:2.7} for each $e>d$ all elements
of $W(e)^o I_\delta W(e)^o$ are consequences of $I_\delta$.
By corollary~\ref{cor:3.9} the product $W(e)^o I_\delta W(e)^o$
equals the sum of the two sided ideals $I_\lm$ where $\lm$ has height
greater than~$n$.  Finally,
by theorem~\ref{thm:3.2} again, this sum is precisely the space of
multilinear identities of degree~$e$, and since we are in characteristic
zero all identities are consequences of multilinear ones and the theorem
follows.
\end{proof}
Using the following lemma we can also describe the mixed \jtr-identties of $M(n,n)$.
\begin{lem} Let $f(x_1,\ldots,x_{n+1})=\sum_\alpha str(x_{n+1}u_\alpha)g_\alpha$ be a
multilinear pure \jtr-identity for $M(n,n)$, where each $u_\alpha$ is a polynomial (without
trace) in $\{J,x_1,\ldots,x_n\}$ and each $g_\alpha$ is a pure \jtr-polynomial.  Then
$F(x_1,\ldots,x_n)=\sum_\alpha u_\alpha g_\alpha$ is a multilinear mixed \jtr-identity for $M(n,n)$
of degree one smaller and $f$ is a consequence of $F$.
\end{lem}
\begin{proof} Since the supertrace is  $E$-linear we have
$f=str(x_{n+1}\sum u_\alpha g_\alpha)$, and the lemma follows from the non-degeneracy of the
supertrace.
\end{proof}
The lemma and theorem combine to give this description of the \jtr-identities of $M(n,n)$.
\begin{thm} All  \jtr\ identities of $M(n,n)$ are consequences of
the identities of degree ${n+2\choose2}-1$.\end{thm}

\end{document}